\documentclass[a4paper,reqno]{amsart}

\usepackage[english]{babel}
\usepackage[utf8x]{inputenc}
\usepackage[T1]{fontenc}

\usepackage[a4paper,top=3cm,bottom=2cm,left=3cm,right=3cm,marginparwidth=1.75cm]{geometry}

\usepackage{amsmath, amsthm, amssymb}
\usepackage{graphicx}
\usepackage[colorinlistoftodos]{todonotes}
\usepackage[colorlinks=true, allcolors=blue]{hyperref}

\usepackage{tikz}
\usepackage{tikz-cd}
\usepackage[all]{xy}

\usepackage{mathtools}

\newtheorem{theorem}{Theorem}[section]
\newtheorem{lemma}[theorem]{Lemma}
\newtheorem{corollary}[theorem]{Corollary}
\newtheorem{proposition}[theorem]{Proposition}

\theoremstyle{definition}
\newtheorem{definition}[theorem]{Definition}
\newtheorem{remark}[theorem]{Remark}
\newtheorem{example}[theorem]{Example}
\newtheorem{notation}[theorem]{Notation}

\numberwithin{equation}{section}

\DeclareMathOperator{\End}{End}

\DeclareMathOperator{\Hom}{Hom}

\DeclareMathOperator{\Trans}{Trans}

\newcommand{\lace}{{\Sigma}}
\newcommand{\lil}{{\mathrm{Int}_\Lambda}}

\newcommand{\id}{\mathrm{id}}

\newcommand{\vect}{\mathrm{vect}}

\newcommand{\suchthat}{\ifnum\currentgrouptype=16 \;\middle|\;\else\mid\fi}

\makeatletter
\let\save@mathaccent\mathaccent
\newcommand*\if@single[3]{%
  \setbox0\hbox{${\mathaccent"0362{#1}}^H$}%
  \setbox2\hbox{${\mathaccent"0362{\kern0pt#1}}^H$}%
  \ifdim\ht0=\ht2 #3\else #2\fi
  }
\newcommand*\rel@kern[1]{\kern#1\dimexpr\macc@kerna}
\newcommand*\widebar[1]{\@ifnextchar^{{\wide@bar{#1}{0}}}{\wide@bar{#1}{1}}}
\newcommand*\wide@bar[2]{\if@single{#1}{\wide@bar@{#1}{#2}{1}}{\wide@bar@{#1}{#2}{2}}}
\newcommand*\wide@bar@[3]{%
  \begingroup
  \def\mathaccent##1##2{%
    \let\mathaccent\save@mathaccent
    \if#32 \let\macc@nucleus\first@char \fi
    \setbox\z@\hbox{$\macc@style{\macc@nucleus}_{}$}%
    \setbox\tw@\hbox{$\macc@style{\macc@nucleus}{}_{}$}%
    \dimen@\wd\tw@
    \advance\dimen@-\wd\z@
    \divide\dimen@ 3
    \@tempdima\wd\tw@
    \advance\@tempdima-\scriptspace
    \divide\@tempdima 10
    \advance\dimen@-\@tempdima
    \ifdim\dimen@>\z@ \dimen@0pt\fi
    \rel@kern{0.6}\kern-\dimen@
    \if#31
      \overline{\rel@kern{-0.6}\kern\dimen@\macc@nucleus\rel@kern{0.4}\kern\dimen@}%
      \advance\dimen@0.4\dimexpr\macc@kerna
      \let\final@kern#2%
      \ifdim\dimen@<\z@ \let\final@kern1\fi
      \if\final@kern1 \kern-\dimen@\fi
    \else
      \overline{\rel@kern{-0.6}\kern\dimen@#1}%
    \fi
  }%
  \macc@depth\@ne
  \let\math@bgroup\@empty \let\math@egroup\macc@set@skewchar
  \mathsurround\z@ \frozen@everymath{\mathgroup\macc@group\relax}%
  \macc@set@skewchar\relax
  \let\mathaccentV\macc@nested@a
  \if#31
    \macc@nested@a\relax111{#1}%
  \else
    \def\gobble@till@marker##1\endmarker{}%
    \futurelet\first@char\gobble@till@marker#1\endmarker
    \ifcat\noexpand\first@char A\else
      \def\first@char{}%
    \fi
    \macc@nested@a\relax111{\first@char}%
  \fi
  \endgroup
}
\makeatother
\newcommand{\wbar}{\widebar}
\newcommand{\wtilde}{\widetilde}

\title{Interleavings and matchings as representations} 
\author[E.G.~Escolar]{Emerson G. Escolar}
\author[K.~Meehan]{Killian Meehan}
\author[M.~Yoshiwaki]{Michio Yoshiwaki}

\address[Emerson G. Escolar]{Center for Advanced Intelligence Project, RIKEN / Institute for Advanced Study, Kyoto University}
\email{emerson.escolar@riken.jp}

\address[Killian Meehan]{Institute for Advanced Study, Kyoto University}
\email{killian.f.meehan@gmail.com}

\address[Michio Yoshiwaki]{Center for Advanced Intelligence Project, RIKEN 
/ Institute for Advanced Study, Kyoto University
/ Osaka City University Advanced Mathematical Institute}
\email{michio.yoshiwaki@riken.jp}


\begin{document}
\begin{abstract}
  
  In order to better understand and to compare interleavings between persistence modules, we elaborate on the algebraic structure of interleavings in general settings.
  In particular, we provide a representation-theoretic framework for interleavings, showing that the category of interleavings under a fixed translation is isomorphic to the representation category of what we call a shoelace. Using our framework, we show that any two interleavings of the same pair of persistence modules are themselves interleaved. Furthermore, in the special case of persistence modules over $\mathbb{Z}$, we show that matchings between barcodes correspond to the interval-decomposable interleavings.
\end{abstract}

\maketitle

\section{Introduction}
\label{sec:introduction}

In recent years, the field of topological data analysis and, in particular, the use of  persistent homology \cite{edelsbrunner2000topological} have grown in popularity.
The algebraic structure of persistent homology can be expressed in the framework of persistence modules, which has led to many generalizations of the structures and methods of persistent homology.

One measure of the distance between two persistence modules is the so-called interleaving distance. The interleaving distance (in certain settings) is defined as the ``smallest'' translation at which interleaving morphisms exist; and interleaving morphisms express a kind of ``approximate'' isomorphism with respect to the corresponding translation. The viewpoint of this work is to treat the interleaving morphisms as objects of study in their own right.

Our work is in part motivated by the paper \cite{bauer2014induced}, in which an isometry theorem is proved between the interleaving distance on (pointwise finite dimensional) persistence modules over $\mathbb{R}$ and the bottleneck distance on the modules' corresponding barcodes. The bottleneck distance is defined by partial matchings of the elements of two barcodes. A partial matching always forms an interleaving of the original persistence modules, and is a ``diagonal'' interleaving between the interval summands. One of the primary goals of this work is to compare  interleavings\textemdash for example, general interleavings from ``diagonal'' interleavings\textemdash even in general settings.


In order to compare arbitrary interleavings, we reuse the concept of interleavings. That is, we define a notion of interleavings between interleavings. This is facilitated by our shoelacing operation, which allows us to realize interleavings as representations of a shoelace proset, on which interleavings can be easily defined. The shoelacing operation can be iterated and allows us to easily talk about interleavings of interleavings, and so on.

By establishing a relationship between the category of interleavings and a representation category (Theorem~\ref{thm:shoelace}), we are able to use known tools in representation theory to study interleavings.
Using this framework, one main result of our work is given in Theorem~\ref{thm:twist_interleaved_interleavings}, which states that any two interleavings (under a fixed translation) of the same pair of persistence modules are themselves interleaved via a translation that is canonically induced by the original translation. Furthermore, in the special case of persistence modules over $\mathbb{Z}$, we show that matchings between barcodes correspond to a special class of interleavings, called interval-decomposable interleavings (Theorem~\ref{thm:corr}). 


In Section~\ref{sec:background}, we review some background definitions that we need. In Section~\ref{sec:shoelaces}, we present our framework of the shoelace proset, which serves as the foundation for
Theorem~\ref{thm:shoelace} stating that interleavings over a fixed translation are essentially representations of the shoelace proset.
In Section~\ref{sec:iterated_shoelaces}, we discuss iterating the shoelacing operation, and prove our main theorem, Theorem~\ref{thm:twist_interleaved_interleavings}. Finally, in Section~\ref{sec:diagonal}, we specialize to the case of the poset $\mathbb{Z}$, and discuss interval-decomposable interleavings and matchings.

We note that the category of $\epsilon$-interleavings (between persistence modules over $\mathbb{R}$) has been defined in \cite{bubenik2014categorification}, where it was directly verified to be abelian.
We contrast this with Theorem~\ref{thm:shoelace}, which implies without extra work that the category of $\epsilon$-interleavings (and more generally, our setting of $\Lambda$-interleavings between $D$-valued persistence modules, Definition~\ref{defn:intcat}) forms an abelian category if $D$ is abelian.
Furthermore, since our result explicitly expresses the category of $\epsilon$-interleavings as a representation category, 
we can utilize the language of representation theory.



\section{Background}
\label{sec:background}
We use the language of category theory in order to express our results. 
For a review of category theory, \cite{kelly1982basic,mac2013categories} are helpful. We adopt the notation and setting of \cite{bubenik2015metrics}.

Recall that a \emph{preordered set} (\emph{proset}) $(P,\leq)$ is a
set $P$ together with relation $\leq$ such that
\begin{itemize}
\item $x \leq x$ for all $x\in P$ (reflexivity), and
\item $x\leq y$ and $y\leq z$ implies $x\leq z$ (transitivity).
\end{itemize}
In what follows, we will simply write $P$ for the proset $(P,\leq)$ where the preorder is understood. 

When talking about a category $C$, we will use the notation $x\in C$ to mean an object of the category.
A proset $(P,\leq)$ can be viewed as a category with objects $x\in P$ and for any objects $x,y$
a unique morphism $x \rightarrow y$ if and only if $x\leq y$, with obvious composition of morphisms and the identity morphism $\id_x=(x \leq x)$ for each $x\in P$. Throughout this work, we view prosets as categories. Furthermore, the unique morphism $x$ to $y$ whenever $x\leq y$ will itself be denoted by $x\leq y$ or $y\geq x$
where convenient. Note that the latter notation enables us to write compositions in a more natural way: $y\leq z$ composed with $x \leq y$ (going from $x$ to $y$ and then from $y$ to $z$) can be written as
\[
  (z \geq y)(y \geq x) = (z \geq x)
\]
since composition of morphisms is usually written ``right-to-left''.

In general, two objects $x$ and $y$ in a category $C$ are said to be \emph{isomorphic} if there exist mutually inverse morphisms.
In the particular case of a proset $P$, this is equivalent to the existence of morphisms
$x\leq y$ and $y\leq x$. Here, the compositions
\[
  (y\geq x)(x\geq y) = (y \geq y) = \id_y
  \text{ and }
  (x\geq y)(y\geq x) = (x \geq x) = \id_x
\]
are automatically the respective identities, and we say that $x$ and $y$ are isomorphic and write $x \cong y$.
Note that a proset need not satisfy the antisymmetry condition, so that $x \leq y$ and $y\leq x$ may hold even if $x \neq y$.

\begin{definition}[Representation of a proset]
  A \emph{representation} \(M\) of a proset \(P\) with
  values in a category \(D\) is a functor \(M:P\to D\).
  A morphism $\eta: M \rightarrow N$ between two representations $M,N:P\rightarrow D$ is a natural transformation.
  The representations \(P\to D\) together with these morphisms and the obvious composition form  
  the category of $D$-valued representations of $P$, denoted \(D^P\).
\end{definition}

To unpack the above definitions, we note that a representation \(M \in D^P\) consists of the following data:
\begin{itemize}
\item an assignment of an object \(M(x)\in D\) to every \(x\in P\), and
\item an assignment of a morphism \(M(x\leq y):M(x)\to M(y)\) to every \(x\leq y\) in \(P\)
\end{itemize}
such that
\begin{enumerate}
\item $M(x\leq x) = \id_{M(x)}$ for all $x\in P$,
\item $M(z \geq y) M(y \geq x) = M(z \geq x)$ whenever $x\leq y$ and $y\leq z$.
\end{enumerate}
A morphism $\eta: M\rightarrow N$ in $D^P$ is a natural
transformation. In particular $\eta$ is a collection \(\{\eta(x):M(x)\to N(x)\}_{x\in P}\)
such that for all \(x\leq y\) in \(P\) the
following diagram commutes:
\begin{equation}
  \label{eq:naturality}
  \begin{tikzpicture}
    \node (mx) at (0,1.5) {$M(x)$};
    \node (my) at (3,1.5) {$M(y)$};
    \node (nx) at (0,0) {$N(x)$};
    \node (ny) at (3,0) {$N(y)\mathrlap{.}$};

    \draw[->] (mx) -- node[midway,above]{$M(x\leq y)$} (my);
    \draw[->] (mx) -- node[midway,left]{$\eta(x)$} (nx);
    \draw[->] (my) -- node[midway,right]{$\eta(y)$} (ny);
    \draw[->] (nx) -- node[midway,below]{$N(x\leq y)$} (ny);
  \end{tikzpicture}
\end{equation}

  We note that in the persistence literature, representations of a
  proset \(P\) are also called \emph{generalized persistence modules}
  over \(P\). In this work, we will simply use the terms \emph{representation} or \emph{functor}.  

\begin{definition}[Translation]
  Let $P$ be a proset.
  \begin{enumerate}
  \item A \emph{translation} of $P$ is a functor $\Lambda:P\rightarrow P$ such that $x \leq \Lambda(x)$ for each $x\in P$. Note that since $\Lambda$ is a functor, if $x\leq y$  
  then $\Lambda(x) \leq \Lambda(y)$.
  \item The \emph{natural transformation $\eta_\Lambda$ of a translation $\Lambda$} of $P$ is the morphism $\eta_\Lambda:1_P\rightarrow \Lambda$ whose morphism at each object $x\in P$ is  $\eta_\Lambda(x) = (x \leq \Lambda(x))$.
  \end{enumerate}
\end{definition}
Given a translation $\Lambda$ of $P$, $\eta_\Lambda$ is well-defined and unique.
For naturality, we need to check the commutativity of Diagram~\eqref{eq:naturality} with $\eta=\eta_\Lambda$, $M = 1_P$, and $N = \Lambda$. These substitutions result in a diagram with terms in $D=P$,
whose commutativity follows automatically from the Thin Lemma below, since $P$ is thin. Recall that a category is said to be thin if there is at most one morphism between any two objects.
\begin{lemma}[{Thin Lemma, \cite[Lemma 3.1]{bubenik2015metrics}}]
In a thin category all diagrams commute, since there
is at most one morphism between any two objects.
\end{lemma}
Note that $D^P$ is thin if $D$ is thin.
In general, we do not assume that $D$ is thin, and so diagrams involving representations
in $D^P$ should not be taken to be automatically commutative.

\begin{definition}
  Let $(P,\leq)$ be a proset. The set $\Trans(P)$ of all translations of $P$ can be given a preorder $\leq$ where   
  $\Lambda \leq \Gamma$ if and only if $\Lambda(i) \leq \Gamma(i)$ for all $i\in P$. 
\end{definition}

Given $M\in D^P$ and a translation $\Lambda$ of $P$, we have a translated representation $M\Lambda \doteq M \circ \Lambda$ by composition. We take note of two types of compositions involving natural transformations and functors below. These compositions are used in the definition of interleavings, so we write them down explicitly.
\begin{definition}
  Let $P$ be a proset, $\phi: M \rightarrow N \in D^P$ be a morphism of representations $M$ and $N$, and let $\Lambda$ be a translation of $P$.
  \begin{enumerate}
  \item Define the morphism
    $M\eta_\Lambda: M \to M\Lambda$ to be the one given by $(M\eta_\Lambda)(x)=M(\eta_\Lambda(x)) = M(x\leq\Lambda(x))$ for all $x\in P$.
  \item Define the morphism
    $\phi\Lambda : M\Lambda \rightarrow N\Lambda$ to be the one given by
    $(\phi\Lambda)(x) = \phi(\Lambda(x))$ for all $x \in P$.
  \end{enumerate}
\end{definition}

We can finally define what we mean by interleavings of representations.
\begin{definition}
  Let $\Lambda$ be a translation of $P$.
  A \(\Lambda\)\emph{-interleaving} of \(M,N\in D^P\)
  is a pair of morphisms
  \(\phi:M\to N\Lambda\), \(\psi:N\to M\Lambda\)
  such that the following diagrams commute:
  \begin{equation}
    \label{eq:interleaving_morphisms}
    \begin{tikzcd}
      M \ar[rr]{}{M\eta_{\Lambda\Lambda}}\ar[dr,swap]{}{\phi} & & M\Lambda\Lambda \\
      & N\Lambda\ar[ur,swap]{}{\psi\Lambda} &
    \end{tikzcd}
    \text{ and }
    \begin{tikzcd}
      & M\Lambda\ar[dr]{}{\phi\Lambda} & \\
      N \ar[rr,swap]{}{N\eta_{\Lambda\Lambda}} \ar[ur]{}{\psi} & & N\Lambda\Lambda \mathrlap{.}
    \end{tikzcd}
  \end{equation}
  Such a pair $(\phi,\psi)$ is called a pair of \emph{$\Lambda$-interleaving morphisms} from $M$ to $N$.
\end{definition}

We note that interleavings can be defined with respect to \emph{two} translations $\Lambda$ and $\Gamma$, giving so-called $(\Lambda,\Gamma)$-interleavings.
For simplicity, we do not treat this generality here.
Next, we give the following definition of the category of
$\Lambda$-interleavings which generalizes the definition of the category of $\epsilon$-interleavings \cite{bubenik2014categorification}, which appear when
$P = (\mathbb{R},\leq)$ and $\Lambda(x) = x+\epsilon$ for some $\epsilon \in \mathbb{R}$.
\begin{definition}
  The category of \(\Lambda\)-interleavings of \(D^P\), denoted
  \(\lil(P,D)\), is the category with objects, morphisms, and composition given as follows.
  \begin{enumerate}
  \item Objects are \(4\)-tuples \((M,N,\phi,\psi)\) with $M,N \in D^P$ and \(\phi:M\to N\Lambda\) and \(\psi:N\to M\Lambda\) is a pair of $\Lambda$-interleaving morphisms.  
    
  \item A morphism from \((M,N,\phi,\psi)\) to \((M',N',\phi',\psi')\) is a pair
    \(g_M:M\to M',g_N:N\to N'\) such that the following diagrams commute:
    \begin{equation}
      \label{eq:morphisms_of_interleavings}
      \begin{tikzcd}
        M \rar{\phi} \dar{g_M} & N\Lambda \dar{g_N\Lambda} \\
        M' \rar{\phi'} & N'\Lambda
      \end{tikzcd}
      \text{ and }
      \begin{tikzcd}
        N \rar{\psi} \dar{g_N} & M\Lambda \dar{g_M\Lambda} \\
        N' \rar{\psi'} & M'\Lambda \mathrlap{.}
      \end{tikzcd}
    \end{equation}
  \item Composition of morphisms is component-wise.    
  \end{enumerate}
  \label{defn:intcat}
\end{definition}


\section{Shoelaces}
\label{sec:shoelaces}

Motivated by interleavings, we introduce the following construction. 
\begin{definition}[Shoelace proset]
  \label{defn:shoelace}
  For a proset \(P\) and a translation \(\Lambda\) on \(P\), define the proset
  \[
    \lace_\Lambda P = P \sqcup P' \doteq \{i\}_{i\in P}\sqcup\{i'\}_{i\in P},
  \] with relation \(x \leq y \)
  defined for $x,y \in P \sqcup P'$ by considering cases for the ordered pair $(x,y)$:
  \begin{enumerate}
  \item Suppose that $(x,y) = (i,j) \in P\times P$ or $(x,y) = (i', j') \in P'\times P'$. If $i \leq j$ in $P$ then $x \leq y$ in $\lace_\Lambda P$.
  \item Suppose that $(x,y) = (i,j') \in P\times P'$ or $(x,y) = (i',j) \in P'\times P$. If $\Lambda i \leq j$ in $P$, then  $x \leq y$ in $\lace_\Lambda P$.  
  \end{enumerate}
  We call this the \emph{$\Lambda$-shoelace proset} of $P$
\end{definition}

That is, we construct $\lace_\Lambda P$ from
two copies of $P$: $P$ and $P'$. The preorder of $\lace_\Lambda P$ restricted to each copy is the same as that of $P$ itself.
To ``go across'' from $P$ to $P'$ or vice-versa, say from element $i$ to $j$ (one of which is in $P$ and the other in $P'$), $i$ and $j$ must satisfy $j \geq \Lambda i$ in $P$.

\begin{notation}
  As a notational shorthand, for $i\in P$ we  
  write $i\in \lace_\Lambda P$ and $i' \in \lace_\Lambda P$ for the two corresponding elements (copies in $\lace_\Lambda P$). Conversely, for $x \in \lace_\Lambda P$, there exists a unique $i\in P$ such that $x = i$ or $x=i'$.
\end{notation}

\begin{proposition}
  Definition~\ref{defn:shoelace} indeed defines a proset $\lace_\Lambda P$.
\end{proposition}
\begin{proof}
  For \(i\in\lace_\Lambda P\), \(i\leq i\) in \(\lace_\Lambda P\) as
  \(i\leq i\) in \(P\). Similarly, \(i'\leq i'\) for any
  \(i'\in\lace_\Lambda P\).

  We split the check for transitivity into multiple cases and take advantage of symmetries.
  \begin{itemize}
  \item If $i\leq j\leq k$ in $\lace_\Lambda P$ (similarly, $i'\leq j'\leq k'$),
    then $i \leq k$ by transitivity in $P$.    
  \item Suppose $i\leq j\leq k'$ (similarly, $i'\leq j'\leq k$).
    Then $\Lambda i\leq \Lambda j\leq k$, and so $i\leq k'$.
  \item Next, suppose $i\leq j'\leq k$ (similarly, $i'\leq j\leq k'$).
    Then $\Lambda\Lambda i\leq \Lambda j\leq k$ in $P$. Furthermore, $i \leq \Lambda i \leq \Lambda\Lambda i$ by definition. Combining, we conclude that $i\leq k$ since $\leq$ is transitive in $P$.
  \item Finally, suppose that  $i'\leq j\leq k$ (similarly, $i'\leq j'\leq k$).
    Then $\Lambda i\leq j\leq k$, and so $i'\leq k$.
  \end{itemize}
\end{proof}

\begin{example}
  Let $P = \{1,2,3\}$ with $1 \leq 2 \leq 3$. Let $\Lambda:P\rightarrow P$ be defined by $\Lambda(i) = \min(i+1,3)$. Then $\lace_\Lambda P$ can be visualized by a Hasse diagram:
  
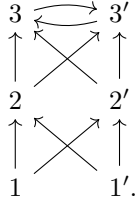
\[
    \begin{tikzcd}
      3 \rar[bend left=15]{} & 3' \lar[bend left=15]\\
      2 \ar[ur]{}{} \uar & 2' \ar[ul]{}{} \uar \\
      1 \ar[ur]{}{} \uar & 1' \mathrlap{.} \ar[ul]{}{} \uar
    \end{tikzcd}
  \]
  Note that while $P$ is a poset (antisymmetry holds), the same is not true for $\lace_\Lambda P$, since $3 \leq 3'$ and $3' \leq 3$ but $3\neq 3'$.  
\end{example}

We note that the above phenomenon is general. Let $i \in P$ be maximal, and $\Lambda$ a translation of $P$. Since $i \leq \Lambda(i)$, we have $i = \Lambda(i)$ and thus $i \leq i'$ and $i \geq i'$ in $\lace_\Lambda P$. Even more generally, we have the following.
\begin{remark}
  Let $P$ be a proset, and $\Lambda$ be a translation. Let $i \in P$. 
  The following are equivalent:
  \begin{enumerate}
  \item $i \cong \Lambda(i)$ in $P$.
  \item $i \cong i'$ in $\lace_\Lambda P$.
  \end{enumerate}
\end{remark}
\begin{proof}
  First, we note that $i \leq \Lambda(i)$ for all $i\in P$. The claim then follows from the fact that $\Lambda(i) \leq i$ in $P$ is equivalent to $i \leq i'$ and $i' \leq i$ in $\lace_\Lambda P$ by case (2) of Definition~\ref{defn:shoelace}.
\end{proof}
In other words, isomorphisms across the two copies of $P$ in $\lace_\Lambda P$  correspond to up-to-isomorphism-fixed-points of $\Lambda$.
To continue our shoelacing analogy, we tie together both sides at any point that does not translate ``upwards''.

Note that $\lace_\Lambda P$ is another proset, and we have the $D$-valued representation category $D^{\lace_\Lambda P}$. The next theorem states that the $D$-valued representation category of the $\Lambda$-shoelace of $P$ is isomorphic to category of $\Lambda$-interleavings of $D^P$.

\begin{theorem}
  \label{thm:shoelace}
  Let $P$ be a proset, and $\Lambda$ a translation of $P$. Then
  \[
    D^{\lace_\Lambda P} \cong \lil(P,D).
  \]
\end{theorem}

\begin{proof}  
  We define an isomorphism \(F:\lil(P,D)\to D^{\lace_\Lambda P}\).   
  
  \paragraph{\textbf{On objects}}  
  We define $F$ to send an object \((M,M',\phi,\psi)\) of $\lil(P,D)$ to the representation \(V\) of
  \(\lace_\Lambda P\), defined as follows:  
  \begin{itemize}    
  \item \(V(i)=M(i)\) and \(V(i')=M'(i)\) for $i,i' \in \lace_\Lambda P$, 
  \item \(V(j\geq i) = M(j\geq i)\) and  \(V(j'\geq i') = M'(j\geq i)\) for all $j \geq i$ in $P$, and
  \item \(V(j'\geq i) = M'(j \geq \Lambda i) \phi(i)\) and \(V(j \geq i') = M(j \geq \Lambda i)\psi(i)\) whenever $j \geq \Lambda i$ in $P$.
  \end{itemize}
  We illustrate $V(j' \geq i)$ and $V(j \geq i')$ as  
  \[
    \begin{tikzcd}
      & M'(j) \\
      & M'(\Lambda i) \ar{u}[swap]{M'(j\geq \Lambda i)} \\
      M(i) \ar{ur}[swap]{\phi(i)} &
    \end{tikzcd}
    \text{ and }
    \begin{tikzcd}
        M(j) & \\
        M(\Lambda i) \uar{M(j \geq \Lambda i)} &  \\
        & M'(i) \ar{ul}{\psi(i)}
      \end{tikzcd}
    \]
  respectively.
  Intuitively, $V$ takes on the values of $M$ and $M'$ respectively on the copies $P$ and $P'$ in $\lace_\Lambda P$, and uses $\phi$ and $\psi$ to ``jump across''.
  
  We check that \(V\) as defined above is indeed a functor (i.e. an element of $D^{\lace_\Lambda P}$.)
  To do so, we first check that $V(fg) = V(f)V(g)$ for $f,g$ morphisms in $\lace_\Lambda P$. We organize our proof by whether or not the morphisms $f$ and $g$ go  
  across (\textbf{swap}) the copies $P$ and $P'$ in $\lace_\Lambda P$ or not (\textbf{stay}).  
  
  \begin{itemize}
  \item \textbf{Stay and stay.} For \(k\geq j\geq i\) in \(\lace_\Lambda P\) (similarly for \(i' \leq j' \leq k' \in \lace_\Lambda P\), but with $M'$ instead of $M$):
    \[
      V(k\geq j) V(j\geq i) = M(k\geq j) M(j\geq i) = M((k\geq j)(j\geq i)) = V((k\geq j)(j\geq i)).
    \]    

  \item \textbf{Swap and swap.} Suppose that \(k\geq j'\geq i\) in \(\lace_\Lambda P\). Then  
  by definition of $\lace_\Lambda P$, $k \geq \Lambda j$ and $j \geq \Lambda i$ in $P$. Since $\Lambda$ is a functor, $\Lambda j \geq \Lambda\Lambda i$. We have
    \[
      \begin{array}{rcl}
        V(k \geq j') V(j' \geq i) &=& [M(k \geq \Lambda j)\psi(j)]  [M'(j \geq \Lambda i)\phi(i)] \\
                                  &=& M(k \geq \Lambda j) M(\Lambda j \geq \Lambda\Lambda i)
                                      \psi({\Lambda i}) \phi(i) \\
                                  &=& M(k \geq \Lambda j) M(\Lambda j \geq \Lambda\Lambda i)
                                      M(\Lambda\Lambda i \geq i) \\
                                  & = & M(k \geq i) \\
                                  & = & V(k \geq i) = V((k\geq j')(j' \geq i)).
      \end{array}
    \]
    This can be understood more readily by the following commutative diagram    
    \[
      \begin{tikzcd}[row sep=small]
        M(k) & \\
        M(\Lambda j) \uar{M(k \geq \Lambda j)} &  \\
        & M'(j) \ar{ul}{\psi(j)} \\
        M(\Lambda\Lambda i) \ar{uu}{(M\Lambda)(j\geq \Lambda i) =  M(\Lambda j \geq \Lambda\Lambda i)} & \\
        & M'(\Lambda i)\mathrlap{,} \ar{uu}[swap]{M'(j\geq \Lambda i)} \ar{ul}{\psi({\Lambda i})} \\
        M(i) \ar{uu}{M\eta_{\Lambda\Lambda}(i) = M(\Lambda\Lambda i \geq i)} \ar{ur}[swap]{\phi(i)} & 
      \end{tikzcd}
    \]
    where the commutativity of the square and triangle (second and third equalities above, respectively) follow from the naturality of $\psi$ and the definition of interleavings.
    
    For \(k'\geq j\geq i'\), the proof is symmetric to the above case.
    
  \item \textbf{Swap then stay.}
    Let \(k \geq j\geq i'\) in $\lace_\Lambda P$. Then, $j \geq \Lambda i$ in $P$, and we have
    \begin{align*}
      V(k \geq i')& = M(k \geq \Lambda i)\psi(i) \\
                  & = M(k \geq j) M(j \geq \Lambda i)\psi(i) \\
                  & = V(k \geq j) V(j \geq i').
    \end{align*}    
    
    The case \(k' \geq j' \geq i\) is similar.
    
  \item \textbf{Stay then swap.} Let \(k'\geq j\geq i\) in $\lace_\Lambda P$. Then  
  $k \geq \Lambda j \geq \Lambda i$ in $P$, and
    \begin{align*}
      V(k'\geq i) & = M'(k \geq \Lambda i)\phi(i) \\
                  & = M'(k \geq \Lambda j) M'(\Lambda j\geq \Lambda i) \phi(i) \\
                  & = M'(k \geq \Lambda j) \phi(j) M(j \geq i) \\
                  & = V(k' \geq j) V(j\geq i)
    \end{align*}
    where the third equality follows from the naturality of $\phi: M\rightarrow M'\Lambda$. We illustrate this computation as the commutative diagram        
    \[
      \begin{tikzcd}[row sep=small]
        & M'(k) \\
        & M'(\Lambda j) \uar[swap]{M'(k\geq \Lambda j)} \\
        M(j) \ar{ur}{\phi(j)} \\
        & M'(\Lambda i)\mathrlap{,} \ar{uu}[swap]{M'(\Lambda j \geq \Lambda i) = (M'\Lambda)(j \geq i)} \\
        M(i) \ar{uu}{M(j \geq i)} \ar{ur}[swap]{\phi(i)}        
      \end{tikzcd}
    \]    
    For \(i'\leq j'\leq k\), the proof is symmetric to the one above.    
  \end{itemize}
  
  There are only $8 = 2^3$ ways to choose unprimed and primed versions of the variables in the inequality $z\geq y\geq x$, and so we have covered all cases. Finally, it is trivial to check that $V(i\geq i) = \id_{V(i)}$ and $V(i' \geq i') = \id_{V(i')}$. Thus, $V$ is indeed a functor.
  
  \paragraph{\textbf{On morphisms}}
  Let $(g_1,g_2) : (M,M',\phi,\psi) \rightarrow (N,N',\bar\phi,\bar\psi)$ be a morphism in $\lil(P,D)$ and let  
  $V_M = F(M,M',\phi,\psi)$
  and
  $V_N = F(N,N',\bar\phi,\bar\psi)$.
  We define $F(g_1,g_2) = g : V_M \rightarrow V_N$ by
  \begin{itemize}
  \item \(g(i):V_M(i)\to V_N(i)\) is the morphism \(g_1(i):M(i)\to N(i)\).
  \item \(g(i'):V_M(i')\to V_N(i')\) is the morphism \(g_2(i):M'(i)\to  N'(i)\).
  \end{itemize}
  for each $i,i' \in \lace_\Lambda P$. In essence, we ``combine'' the two natural transformations $g_1:M\rightarrow N$ and $g_2:M'\rightarrow N'$ into one, so that $g$ restricted to $P$ is $g_1$ and restricted to $P'$ is $g_2$.
  
  We check that, combined this way, $g$ is indeed a natural transformation. That is, for $y \geq x$ in $\lace_\Lambda P$, we have the commutativity of
  \[
    \begin{tikzcd}
      V_M(y) \rar{g(y)} &  V_N(y) \\
      V_M(x) \rar{g(x)}\uar{V_M(y\geq x)} & V_N(x)\mathrlap{.} \uar[swap]{V_N(y \geq x)}
    \end{tikzcd}
  \]
  Restricted to $P$ or $P'$ (that is, $(x,y) = (i,j)$ or $(x,y) = (i',j')$) naturality of $g$  
  follows immediately from that of $g_1$ or $g_2$, respectively.

  Suppose $y = j' \geq i =x$ in $\lace_\Lambda P$. Then $j\geq \Lambda i$ in $P$ and we compute
  \begin{align*}
    V_N(j'\geq i)g(i) & = N'(j \geq \Lambda i) \bar\phi(i) g_1(i) \\
                      & = N'(j \geq \Lambda i)  g_2(\Lambda i) \phi(i) \\
                      & = g_2(j)  M'(j \geq \Lambda i) \phi(i) \\
                      & = g(j') V_M(j' \geq i').
  \end{align*}
  This follows from the commutativity of
  \[
    \begin{tikzcd}
      M'(j) \rar{g_2(j)} & N'(j) \\
      M'(\Lambda i) \rar{g_2(\Lambda i)}\uar{M'(j\geq \Lambda i)}& N'(\Lambda i)\uar[swap]{N'(j\geq\Lambda i)} \\
      M(i) \rar{g_1(i)}\uar{\phi(i)}& N(i)\uar[swap]{\bar\phi(i)} 
    \end{tikzcd}      
  \]
  where the upper square is commutative because $g_2 : M'\rightarrow N'$ is natural, and the lower square is commutative by definition of morphisms in $\lil(P,D)$ as in Diagram~\eqref{eq:morphisms_of_interleavings}.
  The case $y = j \geq i' =x $ can be proved similary.

  \paragraph{\textbf{Functoriality of $F$}} The fact that $F((\id,\id)) = \id$ and $F((h_1,h_2)(g_1,g_2)) = F((h_1g_1,h_2g_2)) = F((h_1,h_2))F((g_1,g_2))$ is clear.  
  Thus $F:\lil(P,D)\to D^{\lace_\Lambda P}$ is indeed a functor.

  Next, we need to construct an inverse for $F$, which we denote by
  $R : D^{\lace_\Lambda P} \rightarrow \lil(P,D)$.

  \paragraph{\textbf{On objects}}
  For $V \in D^{\lace_\Lambda P}$, define
  \[
    R(V) = \left(V|_P, V|_{P'}, \phi(i) = V|_{(\Lambda i)' \geq i}, \psi(i) = V|_{(\Lambda i) \geq i'}\right),
  \]
  where we consider the restriction  $V|_{P'}$ as a representation in $D^P$ by a relabeling: define $V|_{P'}(i) = V(i')$ for $i \in P$ and $V|_{P'}(j \geq i) = V(j' \geq i')$ for $j \geq i$ in $P$.

  The natural transformations $\phi:V|_P \rightarrow V|_{P'}\Lambda$ and $\psi : V|_{P'} \rightarrow V|_P\Lambda$ are defined by
  \[
    \phi(i) = V((\Lambda i)' \geq i)
  \]
  and
  \[
    \psi(i) = V((\Lambda i) \geq i')
  \]
  for $i\in P$.
  To see that $\phi$ is indeed a natural transformation, we check that for $j\geq i$ in $P$,
  \[
    \begin{tikzcd}[column sep=8em]
      V|_P(j) \rar{\phi(j) = V((\Lambda j)' \geq j)} & (V|_{P'}\Lambda)(j) \\
      V|_P(i) \rar{\phi(i) = V((\Lambda i)' \geq i)} \uar{V(j\geq i)} & (V|_{P'}\Lambda)(i) \uar[swap]{(V|_{P'}\Lambda)(j\geq i) = V((\Lambda j)' \geq (\Lambda i)')}
    \end{tikzcd}
  \]
  commutes. This follows from the functoriality of $V$. Similarly, $\psi$ is a natural transformation.

  Finally, we need to check that $\phi:V|_P \rightarrow V|_{P'}\Lambda$ and $\psi : V|_{P'} \rightarrow V|_P\Lambda$ interleaves $V|_P$ and $V|_{P'}$ to verify that $R(V) = \left(V|_P, V|_{P'}, \phi,\psi\right) \in \lil(P,D)$. That is, we need to check the commutativity of diagrams as in Diagram~\eqref{eq:interleaving_morphisms}:
  \[
    \begin{tikzcd}
      V|_P \ar[rr]{}{V|_P\eta_{\Lambda\Lambda}}\ar[dr,swap]{}{\phi} & & V|_P\Lambda\Lambda \\
      & V|_{P'}\Lambda\ar[ur,swap]{}{\psi\Lambda} &
    \end{tikzcd}
    \text{ and }
    \begin{tikzcd}
      & V|_{P}\Lambda\ar[dr]{}{\phi\Lambda} & \\
      V|_{P'} \ar[rr,swap]{}{V|_{P'}\eta_{\Lambda\Lambda}} \ar[ur]{}{\psi} & & V|_{P'}\Lambda\Lambda\mathrlap{.}
    \end{tikzcd}
  \]
  The left diagram at any object $i\in P$ given by
  \[
    \begin{tikzcd}
      V(i) \ar[rr]{}{V(\Lambda\Lambda i \geq i) }\ar[dr,swap]{}{V((\Lambda i)'\geq i)} & & V(\Lambda\Lambda i) \\
      & V( (\Lambda i)') \ar[ur,swap]{}{ V((\Lambda\Lambda i \geq (\Lambda i)'))} &
    \end{tikzcd}
  \]
  clearly commutes by functoriality of $V$.
  The right diagram similarly commutes.

  \paragraph{\textbf{On morphisms}} Let $g : V\rightarrow W$ be a morphism in $D^{\lace_\Lambda P}$. Let us denote $R(V) = (V|_P, V|_{P'}, \phi, \psi)$ and $R(W) = (W|_P, W|_{P'}, \bar\phi, \bar\psi)$. We define a morphism $R(g) : R(V) \rightarrow R(W)$ by setting $R(g) = (g|_P, g|_{P'})$. The fact that $g|_P : V|_P \rightarrow W|_P$ and $g|_{P'} : V|_{P'} \rightarrow W|_{P'}$ are natural transformations follows immediately from naturality of $g$. Finally, we check that $(g|_{P}, g|_{P'})$ satisfies commutativity of diagrams  
  \[
    \begin{tikzcd}
      V|_{P} \rar{\phi} \dar{g|_{P}} & V|_{P'}\Lambda \dar{g|_{P'}\Lambda} \\
      W|_{P} \rar{\phi'} & W|_{P'}\Lambda
    \end{tikzcd}
    \text{ and }
    \begin{tikzcd}
      V|_{P'} \rar{\psi} \dar{g|_{P'}} & V|_{P}\Lambda \dar{g|_{P}\Lambda} \\
      W|_{P'} \rar{\psi'} & W|_{P}\Lambda
    \end{tikzcd}
  \]
  as in Diagram~\eqref{eq:morphisms_of_interleavings}.  
  For each $i \in P$, we have
  \[
    \begin{tikzcd}[column sep=large]
      V(i) \rar{V((\Lambda i)' \geq i)} \dar{g(i)} & V((\Lambda i)') \dar{g((\Lambda i)')} \\
      W(i) \rar{W((\Lambda i)' \geq i)} & W((\Lambda i)')
    \end{tikzcd}
    \text{ and }
    \begin{tikzcd}[column sep=large]
      V(i') \rar{V(\Lambda i \geq i')} \dar{g(i')} & V(\Lambda i) \dar{g(\Lambda i)} \\
      W(i') \rar{W(\Lambda i \geq i')} & W(\Lambda i)
    \end{tikzcd}
  \]
  which are clearly commutative by naturality of $g$. This shows that indeed $R(g) = (g|_{P}, g|_{P'}):R(V) \rightarrow R(W)$ is a morphism in $\lil(P,D)$.

  \paragraph{\textbf{Functoriality of $R$}} Functoriality of $R$ itself is straightforward.  
  $R(\id) = (\id,\id)$ is the identity, and $R(hg) = ((hg)|_{P}, (hg)|_{P'}) = (h|_{P},h|_{P'})(g|_{P},g|_{P'}) = R(h)R(g)$.
  
  \paragraph{\textbf{Inverse}} Finally, we check that $F$ and $R$ are inverses of each other. On morphisms, $RF(g_1,g_2) = (g_1,g_2)$ and $FR(g) = g$.  

  Now let  
  $(M,M',\phi,\psi) \in \lil(P,D)$. We have $RF(M,M',\phi,\psi) = R(V)$, where $V$ is defined as above. Namely,
  \begin{itemize}    
  \item \(V(i)=M(i)\) and \(V(i')=M'(i)\) for $i,i' \in \lace_\Lambda P$, 
  \item \(V(j\geq i) = M(j\geq i)\) and  \(V(j'\geq i') = M'(j\geq i)\) for all $j \geq i$ in $P$, and
  \item \(V(j'\geq i) = M'(j \geq \Lambda i) \phi(i)\) and \(V(j \geq i') = M(j \geq \Lambda i)\psi(i)\).
  \end{itemize}
  The first two conditions clearly imply $V|_P = M$, and $V|_{P'} = M'$.
  Finally,
  \[
    V|_{(\Lambda i)' \geq i)}(i) = M'(\Lambda i \geq \Lambda i) \phi(i) = \phi(i)
  \] and
  \[
    V|_{(\Lambda i) \geq i')}(i) = M(\Lambda i \geq \Lambda i) \psi(i) = \psi(i)
  \] so that
  $R(V) = (M,M',\phi,\psi)$. Thus, $RF$ is the identity.

  In the other direction, suppose that $V \in D^{\lace_\Lambda P}$. Then,
  \[
    FR(V) = F((V|_P, V|_{P'}, \phi = V|_{(\Lambda i)'\geq i)}, \psi = V|_{(\Lambda i)\geq i')}))
  \]
  by definition of $R(V)$. Denote $\bar V = FR(V)$. Then, by definition of $F$,
  \begin{itemize}    
  \item $\bar V(i) = V|_P(i) = V(i)$ and $\bar V(i') = V|_{P'}(i) = V(i')$ for $i,i' \in \lace_\Lambda P$, 
  \item $\bar V(j\geq i) = V|_P(j\geq i) = V(j\geq i)$ and  $\bar V(j'\geq i') = V|_{P'}(j\geq i) = V(j'\geq i')$ for all $j \geq i$ in $P$, and
  \item
    \[
      \begin{array}{rcl}
        \bar V(j'\geq i) &=& V|_{P'}(j \geq \Lambda i) \phi(i) \\
                         &=& V(j' \geq (\Lambda i)') V( (\Lambda i)' \geq i) \\
                         &=& V(j' \geq i)
      \end{array}      
    \]
    and
    \[
      \begin{array}{rcl}
        \bar V(j \geq i') &=& V|_{P}(j \geq \Lambda i) \psi(i) \\
                         &=& V(j \geq (\Lambda i)) V( (\Lambda i) \geq i') \\
                         &=& V(j \geq i').
      \end{array}
    \]    
  \end{itemize}
  This shows that $FR(V) = \bar V = V$. Thus, $FR$ is the identity.

  This completes the proof.  
\end{proof}

The definition of the category of interleavings $\lil(P,D)$ is complicated in that the objects are $4$-tuples $(M,M',\phi,\psi)$ such that $\phi$ and $\psi$ satisfy the interleaving commutativity conditions.
Theorem~\ref{thm:shoelace} states that we can ``package'' these four pieces of data as part of one representation $V = F(M,M',\phi,\psi)$, essentially treating persistence modules $M$ and $M'$ and interleaving morphisms $\phi$ and $\psi$ on the same level. Since $V$ itself is just a representation of the proset $\lace_\Lambda P$, we can now use representation-theoretic tools to study interleavings more directly.


\section{Iterated Shoelaces}
\label{sec:iterated_shoelaces}

With Theorem~\ref{thm:shoelace}, we can now think of $\Lambda$-interleavings of objects in $D^P$ as $D$-valued representations of the proset $\lace_\Lambda P$. From this perspective, the representations being interleaved are given the same footing as the interleaving morphisms, with everything viewed as features of a representation.

This observation enables the following iterated construction. We note that $\lace_\Lambda P$ itself is a proset. Thus, we can consider translations $\Upsilon:\lace_\Lambda P \rightarrow \lace_\Lambda P$, and construct the shoelace of the shoelace: $\lace_\Upsilon(\lace_\Lambda P)$. Again by Theorem~\ref{thm:shoelace}, representations of $\lace_\Upsilon(\lace_\Lambda P)$ can be thought of as $\Upsilon$-interleavings of $\Lambda$-interleavings of $D^P$.

In this section, we study aspects of this iterated construction. First, we start with two special classes of translations of $\lace_\Lambda P$ induced from certain translations $\Gamma:P\rightarrow P$ of the base proset $P$.
\begin{proposition}[Induced translation]
  \label{defn:induced_translation}
  Let $P$ be a proset and \(\Lambda,\Gamma\) be translations on \(P\) such that $\Lambda \Gamma = \Gamma\Lambda$.
  Define $\wbar{\Gamma}$ by
  \begin{enumerate}
  \item $\wbar\Gamma(i) = \Gamma(i)$, $\wbar\Gamma(i') = (\Gamma(i))'$ for $i \in P\subset \lace_\Lambda P$, and 
  \item $\wbar\Gamma(x\leq y) = (\wbar\Gamma(x) \leq \wbar\Gamma(y))$ for $x,y \in \lace_\Lambda P$.
  \end{enumerate}
  Then $\wbar{\Gamma}$ is a translation $\lace_\Lambda P \rightarrow\lace_\Lambda P$, which is called the \emph{translation  induced by} $\Gamma$.  
\end{proposition}
\begin{proof}
  For any $x\leq y$ in $\lace_\Lambda P$, we first check that the unique morphism $\wbar\Gamma (x) \leq \wbar\Gamma (y)$ exists.  
  So suppose that $x \leq y \in \lace_\Lambda P$. We consider four cases depending on where $x$ and $y$ are located in $\lace_\Lambda P$.

  If $(x,y)=(i,j) \in P \times P$ (resp. $(x,y)=(i',j') \in P' \times P'$), we have $\wbar\Gamma (x) = \Gamma (i) \leq \Gamma (j) = \wbar\Gamma (y)$ (resp. $\wbar\Gamma (x) = (\Gamma (i))' \leq (\Gamma (j))' = \wbar\Gamma (y)$ ) since $\Gamma$ is a functor.

  Otherwise,  
  $(x,y) = (i,j') \in P \times P'$ (resp. $(x,y) = (i',j) \in P' \times P$). Since $x = i \leq j' = y$ in $\lace_\Lambda P$ (resp. $x = i' \leq j = y$ in $\lace_\Lambda P$), we have $\Lambda i \leq j$ in $P$. Thus, $\Lambda \Gamma (i) = \Gamma (\Lambda (i)) \leq \Gamma (j) $ in $P$ since $\Gamma \Lambda =\Lambda \Gamma$ and $\Gamma$ is a functor. By definition of $\leq$ in $\lace_\Lambda P$, we have $\wbar\Gamma (x) = \Gamma (i)  \leq (\Gamma (j))' = \wbar\Gamma (y)$ (resp. $\wbar\Gamma (x) = (\Gamma (i))'  \leq \Gamma (j) = \wbar\Gamma (y)$).  Thus $\wbar\Gamma$ can be defined.

It is easy to see that $\wbar\Gamma$ is a functor. Finally, $\wbar\Gamma(i) = \Gamma(i) \geq i$ and $\wbar\Gamma(i') = (\Gamma(i))' \geq (i)'$ are clear. Thus, $\wbar\Gamma$ is indeed a translation of $\lace_\Lambda P$.
\end{proof}

The second type of induced translation we consider is a ``twisted'' translation.

\begin{proposition}[Induced twisted translation]
  \label{defn:twisted_translation}
  Let $P$ be a proset and \(\Lambda,\Gamma\) be translations on \(P\) such that $\Lambda \Gamma = \Gamma\Lambda$ and $\Lambda \leq \Gamma$.
  Define $\wtilde{\Gamma}$ by
  \begin{enumerate}
  \item $\wtilde\Gamma(i) = (\Gamma(i))'$, $\wtilde\Gamma(i') = \Gamma(i)$ for $i \in P\subset \lace_\Lambda P$, and 
  \item $\wtilde\Gamma(x\leq y) = (\wtilde\Gamma(x) \leq \wtilde\Gamma(y))$ for $x,y \in \lace_\Lambda P$.
  \end{enumerate}
  Then $\wtilde{\Gamma}$ is a translation $\lace_\Lambda P \rightarrow\lace_\Lambda P$, which is called the \emph{twisted translation induced} by $\Gamma$.
\end{proposition}
\begin{proof}
Similar to the proof of Proposition~\ref{defn:induced_translation}, 
we first check that  the unique morphism $\wtilde\Gamma (x) \leq \wtilde\Gamma (y)$ exists for any $x\leq y$ in $\lace_\Lambda P$. So, suppose that $x \leq y \in \lace_\Lambda P$.

If $(x,y)=(i,j)\in P\times P$ (resp. $(x,y)=(i',j') \in P' \times P'$), we have $\wtilde\Gamma(i)= (\Gamma(i))' \leq (\Gamma(j))' = \wtilde\Gamma(j)$ (resp. $\wtilde\Gamma(i')= \Gamma(i) \leq \Gamma(j) = \wtilde\Gamma(j')$) since $\Gamma$ is a functor.
Otherwise, $(x,y) = (i,j') \in P \times P'$ (resp. $(x,y) = (i',j) \in P' \times P$).
As above, we have $\Lambda (\Gamma (i)) \leq \Gamma (j) $ in $P$.

By the definition of $\leq$ in $\lace_\Lambda P$, we have $\wtilde\Gamma (x) = (\Gamma (i))' \leq \Gamma (j) = \wtilde\Gamma(y)$ (resp. $\wtilde\Gamma (x) = \Gamma (i) \leq (\Gamma (j))' = \wtilde\Gamma(y)$).
Thus $\wtilde\Gamma$ can be defined.

Finally check that $\wtilde\Gamma$ satisfies $x\leq \wtilde\Gamma(x)$ for all $x \in \lace_\Lambda P$.
Since $\Lambda (i) \leq \Gamma (i)$ for any $i \in P$ and
by definition of $\leq$ in $\lace_\Lambda P$, 
we have $i \leq (\Gamma(i))' = \wtilde\Gamma(i)$ and $(i)' \leq \Gamma (i)= \wtilde\Gamma(i')$. 
Thus, $\wtilde\Gamma$ is indeed a translation of $\lace_\Lambda P$.
\end{proof}

  The main difference between $\wtilde\Gamma$ and $\wbar\Gamma$ is the presence of a ``twist'' in defining the effect of $\wtilde\Gamma$ on objects in $\lace_\Lambda P$. The twisted version $\widetilde{\Gamma}$ sends objects of $P$ to objects of $P'$ and vice-versa.  
  Introducing this twist necessitates the additional condition that $\Lambda \leq \Gamma$ in order to have  $x\leq \wtilde\Gamma(x)$ for all $x \in \lace_\Lambda P$.

Next, we study the composition of these induced translations.
\begin{lemma}
  \label{lem:induced_props}
  Let $P$ be a proset, $\Lambda$ a translation on $P$, and $\Gamma_1,\Gamma_2,\Upsilon_1,\Upsilon_2$ be translations on $P$ that commute with $\Lambda$ such that $\Lambda \leq \Upsilon_1$ and $\Lambda \leq \Upsilon_2$. Then
  \begin{enumerate}
  \item $\wbar{\Gamma_1\Gamma_2} = \wbar{\Gamma_1}\;\wbar{\Gamma_2}$,
  \item $\wbar{\Upsilon_1\Upsilon_2} = \wtilde{\Upsilon_1}\wtilde{\Upsilon_2}$,
  \item $\wtilde{\Gamma_1\Upsilon_2} = \wbar{\Gamma_1}\wtilde{\Upsilon_2}$ and $\wtilde{\Upsilon_1\Gamma_2} = \wtilde{\Upsilon_1}\wbar{\Gamma_2}$.
  \end{enumerate}
\end{lemma}
\begin{proof}
  Statement  
  $(1)$ follows immediately from the definition of $\wbar{\Gamma}$.  
  A direct computation shows statement  
  $(2)$, and  
  composing two twists gives the untwisted version. Statement $(3)$  
  also follows from a similar check.
\end{proof}

We are now able to state our main theorem concerning interleavings being interleaved.
\begin{theorem}
  \label{thm:twist_interleaved_interleavings}
  Let $M,N \in D^P$. Any two $\Lambda$-interleavings $(M,N,\phi,\psi)$ and $(M,N,\phi',\psi')$ of $M$ and $N$ are $\wtilde\Lambda$-interleaved.
\end{theorem}
\begin{proof}
  Let $V = (M,N,\phi,\psi)$ and $V' = (M,N,\phi',\psi')$ and view $V,V'$ as objects of $D^{\lace_\Lambda P}$ via Theorem~\ref{thm:shoelace}. We simply need to provide a pair $\Phi$, $\Psi$ of $\wtilde\Lambda$-interleaving morphisms between $V$ and $V'$.

  We will  
  define $\Phi: V\rightarrow V'\wtilde\Lambda$, a morphism between representations of $\lace_\Lambda P$. Restricted to $P$, we note that $V|_P = M$ and $(V'\wtilde\Lambda) |_P = N\Lambda$, and restricted to $P'$, $V|_{P'} = N$ and $(V'\wtilde\Lambda) |_{P'} = M\Lambda$. So, we choose $\Phi = (\phi, \psi')$. That is, $\Phi|_P = \phi$ and $\Phi|_{P'} = \psi'$.
  
  We represent this by the following diagram, which only makes sense ``up to appropriate $\Lambda$-composition'', where we are suppressing $\Lambda$ shifts in order to have one diagram  
  \[
    \begin{tikzcd}
      & \small{P} & \small{P'} \\[-1.8em]
      V:\dar[swap]{\Phi=(\phi,\psi')}
      &
      M \rar[shift left]{\phi}\dar{\phi} & N \lar[shift left]{\psi} \dar{\psi'}      
      \\
      V'\wtilde\Lambda:
      &
      N \rar[shift left]{\psi'} & M\mathrlap{.} \lar[shift left]{\phi'}
    \end{tikzcd}
  \]
   Note that the bottom row, representing $V'\wtilde\Lambda$, has the places of $M$ and $N$ transposed. This comes from the use of the twisted induced translation.
  To be precise, we expand out the above diagram as
  \[
    \begin{tikzcd}
      & \small{P} & \small{P'} & \small{P} & \small{P'}\\[-1.8em]
      V:\dar[swap]{\Phi=(\phi,\psi')}
      &
      M \rar{\phi}\dar{\phi} & N\Lambda \dar{\psi'\Lambda}
      &
      M\Lambda \dar{\phi\Lambda} & N \lar{\psi} \dar{\psi'}
      \\
      V'\wtilde\Lambda:
      &
      N\Lambda \rar{\psi'\Lambda} & M\Lambda\Lambda
      &
      N\Lambda\Lambda & M\Lambda\mathrlap{.} \lar{\phi'\Lambda}
    \end{tikzcd}
  \]
  
  We need to check the commutativity of both diagrams in order to show that $\Phi$ is a morphism. The left diagram is clearly commutative. Commutativity of the right diagram follows from the fact that  
  \[
    (\phi\Lambda)(\psi) = N\eta_{\Lambda\Lambda} = (\phi'\Lambda)(\psi')
  \]
  by the definition of interleavings. Thus, $\Phi: V\rightarrow V'\wtilde\Lambda$ is indeed a morphism of $\lace_\Lambda P$ representations.

  In the opposite direction, and by a similar analysis, we choose $\Psi = (\psi, \phi')$, schematically represented ``up to appropriate $\Lambda$-composition'' by  
  \[
    \begin{tikzcd}
      & \small{P} & \small{P'} \\[-1.8em]
      V\wtilde\Lambda:
      &
      M \rar[shift left]{\phi} & N \lar[shift left]{\psi} 
      \\
      V'\uar{\Psi=(\psi,\phi')}:
      &
      N \rar[shift left]{\psi'}\uar{\psi} & M \lar[shift left]{\phi'}\uar{\phi'}
    \end{tikzcd}    
  \]
  which means  
  \[
    \begin{tikzcd}
      & \small{P} & \small{P'} & \small{P} & \small{P'}\\[-1.8em]
      V\wtilde\Lambda:
      &
      M\Lambda \rar{\phi\Lambda} & N\Lambda\Lambda
      &
      M\Lambda\Lambda & N\Lambda \lar{\psi\Lambda}
      \\
      V' \uar{\Psi=(\psi,\phi')}:
      &
      N \rar{\psi'} \uar{\psi} & M\Lambda \uar{\phi'\Lambda}
      &
      N\Lambda \uar{\psi\Lambda} & M\mathrlap{.} \lar{\phi'} \uar{\phi'}
    \end{tikzcd}
  \]
  
  Finally,  
  \[
    (\Psi\wtilde\Lambda)\Phi = V\eta_{\wtilde\Lambda\wtilde\Lambda}
  \] and
  \[
    (\Phi\wtilde\Lambda)\Psi = V'\eta_{\wtilde\Lambda\wtilde\Lambda}
  \] follows immediately by restricting to $P$ and $P'$ where the equalities hold by the definition of $\Phi$ and $\Psi$.  
\end{proof}

\begin{remark}
  Theorem~\ref{thm:twist_interleaved_interleavings} is optimal in the following sense. In order to define $\wtilde\Gamma$, we need the condition $\Lambda \leq \Gamma$. Thus, the statement for $\Gamma=\Lambda$ in Theorem~\ref{thm:twist_interleaved_interleavings} cannot be improved upon with regard to  
  $\wtilde\Gamma$-interleavings of the given $\Lambda$-interleavings.
\end{remark}

In order to ``remove'' the twist in Theorem~\ref{thm:twist_interleaved_interleavings}, we use the following observation. 
\begin{lemma}
  \label{lem:worse_interleaving}
  Let $\Lambda, \Gamma$ be translations of a proset $P$, and let $M,N \in D^P$ be $\Lambda$-interleaved. If $\Lambda \leq \Gamma$, then $M$ and $N$ are $\Gamma$-interleaved.
\end{lemma}
\begin{proof}
By assumption, we have a pair of $\Lambda$-interleaving morphisms
  \(\phi:M\to N\Lambda\), \(\psi:N\to M\Lambda\). 
Since $\Lambda \leq \Gamma$, we have a morphism $\xi: \Lambda \to \Gamma$. 
Thus, we can define a pair of morphisms 
   \(\wtilde\phi:M\to N\Gamma\), \(\wtilde\psi:N\to M\Gamma\) 
as $\wtilde\phi:=N(\xi)\circ\phi$ and $\wtilde\psi:=M(\xi)\circ\psi$.  
Then it is clear that the following diagrams commute:
\[
    \begin{tikzcd}
      M \ar[rr]{}{M\eta_{\Gamma\Gamma}}\ar[dr,swap]{}{\wtilde\phi} & & M\Gamma\Gamma \\
      & N\Gamma\ar[ur,swap]{}{\wtilde\psi\Gamma} &
    \end{tikzcd}
    \text{ and }
    \begin{tikzcd}
      & M\Gamma\ar[dr]{}{\wtilde\phi \Gamma} & \\
      N \ar[rr,swap]{}{N\eta_{\Gamma\Gamma}} \ar[ur]{}{\wtilde\psi} & & N\Gamma\Gamma\mathrlap{.}
    \end{tikzcd}
\]
\end{proof}

Lemma~\ref{lem:worse_interleaving}, together with the observation that twisting twice gives an untwisted induced translation, gives the following ``untwisted version'' of Theorem~\ref{thm:twist_interleaved_interleavings}.
\begin{corollary}
  \label{cor:interleaved_interleavings}
  Let $M,N \in D^P$. Any two $\Lambda$-interleavings $(M,N,\phi,\psi)$ and $(M,N,\phi',\psi')$ of $M$ and $N$ are $\wbar{\Lambda\Lambda}$-interleaved.
\end{corollary}
\begin{proof}
  Let $V = (M,N,\phi,\psi)$ and $V' = (M,N,\phi',\psi')$ and view $V,V'$ as objects of $D^{\lace_\Lambda P}$ via Theorem~\ref{thm:shoelace}. By Theorem~\ref{thm:twist_interleaved_interleavings}, $V$ and $V'$ are $\wtilde\Lambda$-interleaved.
  Note that $\wtilde\Lambda \leq \wtilde\Lambda\wtilde\Lambda = \wbar{\Lambda\Lambda}$ by Lemma~\ref{lem:induced_props}, and thus by Lemma~\ref{lem:worse_interleaving}  
  $V$ and $V'$ are $\wbar{\Lambda\Lambda}$-interleaved.
\end{proof}


\section{Interval-decomposable interleavings and matchings}
\label{sec:diagonal}

In the rest of this section, we specialize to the poset $P = (\mathbb{Z},\leq)$ and the target category $D = \vect_K$. Furthermore, we restrict our attention to what we call $\epsilon$-uniform translations ($\Lambda_\epsilon$) with respect to certain height functions. We then study a special class of $\Lambda_\epsilon$-interleavings called the interval-decomposable interleavings and provide a direct relationship with $\epsilon$-matchings. 

First, we provide the following definitions. A \emph{height function} on a poset $P$ is a monotone function $h: P\rightarrow \mathbb{R}$. That is, for $x\leq y$ in $P$, $h(x) \leq h(y)$.
For a translation $\Lambda:P\to P$,
\begin{enumerate}
    \item 
    $\Lambda$ is said to be {\em $\epsilon$-uniform (with respect to $h$)} if $h(\Lambda(x)) - h(x) = \epsilon$ for any $x\in P$.
    \item 
    $\Lambda$ has {\em height $\epsilon$ (with respect to $h$)} if $\sup_{x\in P}(h(\Lambda(x)) - h(x)) = \epsilon$.
  \end{enumerate}
  Note that any $\epsilon$-uniform translation has height $\epsilon$.

Throughout the rest of this work, we fix the following choices of height functions.
We associate to the poset $\mathbb{Z}$ the canonical height  function that is the inclusion $h_\iota: \mathbb{Z} \hookrightarrow \mathbb{R}$. For any translation $\Gamma$ of $\mathbb{Z}$, the associated shoelace $\lace_\Gamma \mathbb{Z}$ is given the canonical height $h$ function defined by $h(x) := h_\iota(x) =x $ and $h(x') := h_\iota(x) = x$ for $x \in \mathbb{Z}$.

We then define the translations $\Lambda_\epsilon$ of $\mathbb{Z}$ for $\epsilon$ any nonnegative integer.
\begin{definition}[The $\epsilon$-uniform translation $\Lambda_\epsilon$ of $\mathbb{Z}$]
For $\epsilon$ any nonnegative integer, there is a unique $\epsilon$-uniform translation $\Lambda_\epsilon: \mathbb{Z}\rightarrow \mathbb{Z}$ given by $x\mapsto x+\epsilon$. Under this translation, we simplify the notation and write $\lace_\epsilon \mathbb{Z} := \lace_{\Lambda_\epsilon}\mathbb{Z}$.
\end{definition}

Next, we see that the uniformness of $\Lambda_\epsilon$ is well-behaved under the induction discussed in Section~\ref{sec:iterated_shoelaces}.
\begin{lemma}
  Let $\epsilon \geq 0$ and recall that $\Lambda_\epsilon:\mathbb{Z}\rightarrow \mathbb{Z}$ is the $\epsilon$-uniform translation on $\mathbb{Z}$.
  If $\Gamma$ is a translation on $\mathbb{Z}$ such that $\Gamma\Lambda_\epsilon = \Lambda_\epsilon \Gamma$, then
  \begin{enumerate}  
  \item the translation $\wbar{\Lambda_\epsilon}$ of $\lace_\Gamma \mathbb{Z}$ is $\epsilon$-uniform, and 
  \item the translation $\wtilde{\Lambda_\epsilon}$ of $\lace_\Gamma \mathbb{Z}$  is $\epsilon$-uniform (if $\Gamma \leq \Lambda_\epsilon$, so that $\wtilde{\Lambda_\epsilon}$ is well-defined).
  \end{enumerate}
  \label{lem:induced_uniformness}
\end{lemma}
\begin{proof}
This follows immediately from the definitions.
\end{proof}

Using the above definitions, we are able to state the following Remark~\ref{rem:twist_rephrase}, which translates Theorem~\ref{thm:twist_interleaved_interleavings} into a statement in terms of ``$\epsilon$-interleavings'' (interleavings with respect to an $\epsilon$-uniform translation). While we do not pursue this connection further, we note that the $\epsilon$-interleavings are used in usual definitions of the interleaving distance \cite{chazal2009proximity,bauer2014induced}.
\begin{remark}
  Still in the case of $P =\mathbb{Z}$, and rephrasing Theorem~\ref{thm:twist_interleaved_interleavings}, we see from Lemma~\ref{lem:induced_uniformness} that every pair of interleavings with respect to an $\epsilon$-uniform translation (which can only be $\Lambda_\epsilon$) are themselves interleaved by an $\epsilon$-uniform translation (given by $\wtilde{\Lambda_\epsilon}$).
  \label{rem:twist_rephrase}
\end{remark}

Next, we turn our attention to $\epsilon$-matchings and their relationship to what we call interval-decomposable interleavings. We recall the following definitions.
\begin{definition}
  Let $P$ be a poset and $S$ a subposet of $P$.
  \begin{enumerate}
  \item 
    A subposet $S$ is said to be {\em connected} if $S=S_1 \bigsqcup S_2$ such that $s_1$ and $s_2$ are not comparable for any $s_1\in S_1$ and any $s_2 \in S_2$ implies $S_1 = \emptyset$ or $S_2 =\emptyset$. 
  \item
    A subposet $S$ is said to be {\em convex} if for any $s,t\in S$ with $s\leq t$, the segment in $P$ 
    \[
      [s,t] = \{x \in P \suchthat s\leq x \leq t\}
    \]
    is a subposet of $S$.
  \item
    A subposet $S$ is called {\em an interval} if it is connected and convex.
\end{enumerate}
\end{definition}

\begin{definition}
  A representation $M$ of a poset $P$ is called {\em an interval} if:
  \begin{enumerate}
      \item it is thin, namely, $\dim M(x) \leq 1$ for all $x \in P$,
      \item its support subposet $\textrm{supp}(M) = \{x\in P\suchthat M(x)\neq 0\}$ is an interval, and
      \item $M(x\leq y)=1$ for any comparable pair of $x,y $ in the support subposet $S$.
  \end{enumerate}
A representation $M$ of $P$ is said to be {\em interval-decomposable} if it is isomorphic to a direct sum of interval representations.
\end{definition}

\begin{remark}
  \leavevmode
  \begin{enumerate}
  \item Note that the endomorphism algebra $\End_{(\vect_K)^P} (M)$ of an interval representation $M$ is just $K$, and hence any interval representation is indecomposable.    
    Moreover, by the Krull-Schmidt-Remak-Azumaya theorem, 
    any interval-decomposable representation has indecomposable decomposition unique up to isomorphism and permutation of terms.
  \item Crawley-Boevey proved that any pointwise finite persistence module is interval-decomposable for $P=\mathbb{R}$ \cite{MR3323327}. Thus, an analogous statement is true for $P=\mathbb{Z}$. Note that any interval representation of $P=\mathbb{Z}$ is in one of the following forms:    
    \[
      \begin{tikzcd}[row sep=0pt, column sep=1.5em]
          &  & & \scriptstyle{\text{at }x} &  & \scriptstyle{\text{at }y}  &  & \\[-5pt]
        I[x,y]: & \cdots \rar & 0 \rar & K \rar & \cdots \rar & K \rar & 0 \rar & \cdots, \\
        I(-\infty, y]: & \cdots \rar & K \rar & K \rar & K \rar & K \rar & 0 \rar & \cdots, \\ 
        I[x,\infty):& \cdots \rar & 0 \rar & K \rar & K \rar & K \rar & K \rar & \cdots, \\ 
        I(-\infty,\infty):& \cdots \rar & K \rar & K \rar & K \rar & K \rar & K \rar & \cdots.
      \end{tikzcd}
    \]
  \end{enumerate}
\end{remark}

For a persistence module $M$, we denote by $B(M)$ its \emph{barcode}. Namely, $B(M)$ is the multiset of the (isomorphism classes of) interval direct summands $I[x,y],I(-\infty,y],I[x,\infty), I(-\infty,\infty)$ in an indecomposable decomposition of $M$.
In order to simplify the statements of the following definitions and results, we adopt the following convention:
for any integer $x$,
\[
  \begin{array}{rcl}
    |\pm \infty - (\mp \infty)| &=& \infty,\\
    |\pm \infty - (\pm \infty)| &=& 0,\\
    |\pm \infty -x| &=& \infty, \text{ and }\\
    |x-(\pm \infty)| &=& \infty.
  \end{array}
\]

Let us recall the definition of an $\epsilon$-matching \cite{bauer2014induced}.
\begin{definition}[$\epsilon$-matching]
  Let $M,N$ be persistence modules.
  An $\epsilon$-matching $\sigma:B(M)\to B(N)$
  is a matching $\sigma$ (namely, it gives a bijection between submultisets) such that  
  \begin{enumerate}
  \item if $\sigma(I)=J$ with $I=I[x,y], J=I[s,t]$, then  
    $|x-s|\leq \epsilon,\ |y-t|\leq \epsilon$;
    and 
  \item if $I=I[x,y]$ is unmatched, then $|x-y|<2\epsilon$.
  \end{enumerate}
  \label{def:ematching}
\end{definition}
To see how our convention for infinite intervals interacts with the definition, we see for example that
$
  \sigma([1,\infty)) = [0,\infty)
$
is valid for a $1$-matching, since $|1-0| \leq 1$ and $|\infty - \infty| = 0 \leq 1$. On the other hand,
$
  \sigma([1, 10000)) = [0,\infty)
$
is not valid for a $1$-matching, since $|10000-\infty| = \infty \not\leq 1$. Generally, an infinite interval can only be matched to another infinite interval of the same ``type'', and infinite intervals cannot be unmatched.

Condition (2) of Definition~\ref{def:ematching} states that unmatched intervals must be short. However, there is no restriction in general that short intevals always be unmatched. This presents some technical problems for our main Theorem~\ref{thm:corr}, and so we add the  following condition on $\epsilon$-matchings to state our correspondence between $\epsilon$-matchings and interval-decomposable interleavings. Essentially, we restrict what sort of short intervals can be matched by $\sigma$.

\begin{definition}[essential $\epsilon$-matching]
  Let $M,N$ be persistence modules.
  A $\epsilon$-matching $\sigma : B(M)\to B(N)$ is said to be {\em essential} 
  if for every pair of intervals $I[x,y]$, $I[s,t]$ with $|y-x|< 2\epsilon$, $|t-s|<2 \epsilon$ such that  $\sigma(I[x,y]) = I[s,t]$, the following Condition \eqref{eq:conast} holds
  \begin{equation} \label{eq:conast}    
    s-\epsilon \leq x \leq t-\epsilon \leq y \text{ or } x-\epsilon \leq s \leq y-\epsilon \leq t.
  \end{equation}
  \label{def:essential}
\end{definition}

Let us rephrase Condition~\eqref{eq:conast} for matched short intervals in more algebraic terms.
\begin{lemma} \label{lem:ess}
  Let $I[x,y]$, $I[s,t]$ be short intervals ($|y-x|< 2\epsilon$, $|t-s|<2\epsilon$) with $|x-s|\leq \epsilon,\ |y-t|\leq \epsilon$. Then, the following are equivalent.
  \begin{enumerate}
  \item Condition~\eqref{eq:conast} holds: $s-\epsilon \leq x \leq t-\epsilon \leq y \text{ or } x-\epsilon \leq s \leq y-\epsilon \leq t$.
  \item $\Hom (I[x,y],I[s,t]\Lambda_\epsilon)\not=0 \text{ or }\Hom (I[s,t],I[x,y]\Lambda_\epsilon)\not=0$.
  \item There exists a nontrivial $\Lambda_\epsilon$-interleaving between $I[x,y]$ and $I[s,t]$.
  \end{enumerate}  
\end{lemma}
By `nontrivial interleaving' we mean an interleaving pair of morphisms $(\phi,\psi)$ such that at least one of $\phi$, $\psi$ is nonzero.
\begin{proof}
First, we note that (1) and (2) are equivalent for intervals in general, and not just for short matched intervals.

(3) $\implies$ (2) This is immediate.

(1,2) $\implies$ (3)
At least one of $f:I[x,y] \to I[s,t]\Lambda_\epsilon$ or 
$g:I[s,t]\to I[x,y]\Lambda_\epsilon$ defined by 
\[
  f(a)=
  \left\{
    \begin{array}{cc}
      \id     &  (x\leq a \leq t-\epsilon)\\
      0     & \text{otherwise,}
    \end{array}
  \right.
  \text{ or }
  g(a)=
  \left\{
    \begin{array}{cc}
      \id     &  (s\leq a \leq y-\epsilon)\\
      0     & \text{otherwise,}
    \end{array}
  \right.
\]
is nonzero. 
These morphisms always fit into the commutative diagrams
\[
  \begin{tikzcd}
    I[x,y] \ar[rr]{}{0}\ar[dr,swap]{}{f} & & I[x-2\epsilon,y-2\epsilon]\\
    & I[s-\epsilon, t-\epsilon] \ar[ur,swap]{}{g\Lambda_\epsilon} &
  \end{tikzcd}
\]
and
\[
  \begin{tikzcd}
    & I[x-\epsilon,y-\epsilon]\ar[dr]{}{f\Lambda_\epsilon} & \\
    I[s,t] \ar[rr,swap]{}{0} \ar[ur]{}{g} & & I[s-2\epsilon, t-2\epsilon]
  \end{tikzcd}
\]
where both morphisms $I[x,y] \to I[x-2\epsilon,y-2\epsilon]$ and $I[s,t] \to I[s-2\epsilon, t-2\epsilon]$ are $0$ because both intervals have length $<2\epsilon$.

Thus, $(f,g)$ forms a nontrivial $\Lambda_\epsilon$-interleaving between $I[x,y]$ and $I[s,t]$. 
\end{proof}


Next is the main result of this section.
\begin{theorem} \label{thm:corr}
Let $M,N$ be persistence modules $\mathbb{Z} \to \vect_K$.
There is a bijective corespondence between the collection of essential $\epsilon$-matchings $\sigma : B(M) \to B(N)$ and the set of isoclasses of interval-decomposable representations $L$ of $\lace_\epsilon\mathbb{Z}$ such that $L|_{\text{left}}=M$ and $L|_{\text{right}}=N$.
\end{theorem}

To prove Theorem~\ref{thm:corr} we first show the following Lemmas.

\begin{lemma} \label{lem:corronlyif}
For an interval representation $L$ of $\lace_\epsilon \mathbb{Z}$, $I:=L|_{\text{left}}$ and $J:=L|_{\text{right}}$
are interval representations of $\mathbb{Z}$.
Moreover, suppose that $I = I[x,y]$ and $J =I[s,t]$, both non-empty intervals.
Then we have
\[
  |x-s|\leq \epsilon, \;\;
  |y-t|\leq \epsilon.
\]
If, in addition, $|y-x|< 2\epsilon$ and $|t-s|<2\epsilon$, then Condition \eqref{eq:conast} holds.
\end{lemma}
\begin{proof}
  The first statement follows easily from the definitions.
  
  Suppose to the contrary that $|x-s|> \epsilon$.
  Then $x>x-\epsilon>s$ or $x<s-\epsilon<s$. We handle these two cases below and show that they lead to a contradiction. Note that it is only possible to get $|x-s| = \infty> \epsilon$ according to our convention when either $x = -\infty$ or $s=-\infty$ but not both.
  
  \begin{enumerate} 
  \item Suppose that $x>s- \epsilon>s$ and that both $x$ and $s$ are finite. Then $J(s')\to J((x-\epsilon)')\to I(x)$ is non-zero but $J(s') \to I(x-\epsilon)$ and $I(x-\epsilon)\to I(x)$ are zeros 
  since $I(x-\epsilon)=0$ by definition.
  This contradicts the commutativity of the maps of $L$.
  
  In the case that $s = -\infty$, we simply replace $s$ by a sufficiently large negative integer and the above argument works as-is.
  
  \item Let us suppose $x<s-\epsilon<s$ and assume that both $x$ and $s$ are finite. Then, 
  $I(x)\to I(s-\epsilon)\to J(s')$ is non-zero but $I(x) \to J((s-\epsilon)')$ and $J((s-\epsilon)')\to J(s')$ are zeros
  since $J((s-\epsilon)')=0$ by definition.
  By the commutativity in $L$, this is also a contradiction.
  
  In the case that $x=-\infty$, a similar argument as above works.
  \end{enumerate}
  
  Symmetrically, the assumption that $|y-t|>\epsilon$ leads to a contradiction.
  
  Finally, suppose in addition that $|y-x|< 2\epsilon$, $|t-s|<2\epsilon$. If the pair of $I, J$ does not satisfy Condition \eqref{eq:conast}, then $L$ can be decomposed as $L_{I,0} \oplus L_{0,J}$ with $L_{I,0}|_{\text{left}} = I$, $L_{I,0}|_{\text{right}} = 0$ and $L_{0,J}|_{\text{left}} = 0$, $L_{0,J}|_{\text{right}} = J$. This follows from the fact that  all $\epsilon$-interleaving morphisms between $I$ and $J$ are trivial by Lemma~\ref{lem:ess}. This contradicts the indecomposability of $L$, thus showing that Condition~\eqref{eq:conast} must be satisfied.
\end{proof}

\begin{lemma} \label{lem:corrif}
  Let $I=I[x,y]$, $J=I[s,t]$ be $\epsilon$-matched ($|x-s|\leq \epsilon,\ |y-t|\leq \epsilon$). If
  \begin{enumerate}
  \item both are short intervals ($|y-x|< 2\epsilon$ and $|t-s|< 2\epsilon$) and satisfy Condition \eqref{eq:conast}, or
  \item $|y-x| \geq 2\epsilon$ or $|t-s| \geq 2\epsilon$,
  \end{enumerate} 
  then we can construct an interval representation $L$ of $\lace_\epsilon \mathbb{Z}$ such that $L|_{\text{left}}=I$ and $L|_{\text{right}}=J$.
\end{lemma}
\begin{proof}
  In the first case, the conclusion follows from the proof of Lemma~\ref{lem:ess} where it can be checked that the quadruple
  $(I,J, f,g)$
  corresponds to an interval representation of $\lace_\epsilon \mathbb{Z}$ via Theorem~\ref{thm:shoelace}.
  
  In the second case, we may assume that $|y-x|>2\epsilon$ and $|y-x|\geq |t-s|$.
  Then,
  at least one of the conditions 
  $x\leq s \leq x+\epsilon \leq y-\epsilon \leq t \leq y$, 
  $s\leq x \leq x+\epsilon \leq y-\epsilon \leq t \leq y$
  or 
  $x\leq s \leq x+\epsilon \leq y-\epsilon \leq y \leq t$ holds. 
  In any case,
  for any $z$ with $x \leq z \leq z+2 \epsilon \leq y$, we have $J(z+\epsilon)\not =0$ since $s\leq z+\epsilon \leq t$.
  Thus,
  it can be cheched that the quadruple $(I,J,f,g)$
  in the proof of Lemma~\ref{lem:ess}
  corresponds to an interval representation of $\lace_\epsilon \mathbb{Z}$ via Theorem~\ref{thm:shoelace} 
  since $I(z) \to J(z+\epsilon)\to I(z+2 \epsilon)=I(z)\to I(z+2 \epsilon)$ for $x\leq z \leq z+2\epsilon\leq y$. 
\end{proof}

\begin{proof}[Proof of Theorem~\ref{thm:corr}]
Let us construct mutually inverse bijective maps  $F$ and $G$ between the collection of essential $\epsilon$-matchings $\sigma:B(M)\to B(N)$ and the set of isoclasses of interval decomposable representations $L$ of $\lace_{\epsilon} \mathbb{Z}$ such that $L|_{left}=M$ and $L|_{right}=N$.
\begin{itemize}
\item Let $\sigma: B(M)\rightarrow B(N)$ be an essential $\epsilon$-matching. We construct the corresponding interval-decomposable interleaving:
  \[
    F(\sigma):=(\bigoplus_{\sigma(I)=J}L_{I,J}) \oplus (\bigoplus_{I \in B(M) \text{, unmatched}} L_{I,0} )\oplus (\bigoplus_{J\in B(N)\text{, unmatched}} L_{0,J} ),
  \]
  where the intervals $L_{I,J}$, $L_{I,0}$, and $L_{0,J}$ are defined as below.

  For each pair $I=I[x,t]$, $J=I[s,t]$ with $\sigma(I)=J$, the hypothesis of Lemma~\ref{lem:corrif} is satisfied since $\sigma$ is essential. Thus, we obtain the interval representation $L_{I,J}$ of $\lace_\epsilon \mathbb{Z}$ such that $L_{I,J}|_{\text{left}}=I$ and $L_{I,J}|_{\text{right}}=J$ by Lemma~\ref{lem:corrif}.

  For each $I=I[x,y] \in B(M)$ unmatched, we construct the interval representation $L_{I,0}$  such that $L_{I,0}|_{\text{left}}=I$ and $L_{I,0}|_{\text{right}}=0$. Note that this is a valid representation, since $|y-x| < 2\epsilon$. We do an analogous construction of $L_{0,J}$ for each $J \in B(N)$ unmatched.

\item In the other direction, let $L$ be an interval-decomposable representation of $\lace_\epsilon \mathbb{Z}$ such that $L|_{left}=M$ and $L|_{right}=N$.
  We define an essential $\epsilon$-matching $G(L):=\sigma_L$ below.

  Without loss of generality, since we are working up to isoclass,
  \[
    L=\bigoplus_{V\text{:interval}} V.
  \]
  For each $V$, an interval direct summand of $L$ appearing in the above decomposition, we define the following.
  We set $I:=V|_{\text{left}}$ and $J:=V|_{\text{right}}$.
  It is obvious that $I\in B(M)$ and $J \in B(N)$.

  If $J=0$, then we define $\sigma_L$ so that $I$ is unmatched.
  Let us check that $I$ has length $<2\epsilon$.
  Indeed, if the length of $I=I[x,y]$ is greater than or equal to $2\epsilon$, then the internal map of $I(x) = V(x) \rightarrow V(x+2\epsilon) = I(x+2\epsilon)$ of $V$is nonzero. However, the maps
  from $V(x)$ to $V((x+\epsilon)')$ and from $V((x+\epsilon)')$ to $V(x+2\epsilon)$ are both zero since $V((x+\epsilon)') = J(x+\epsilon) = 0$. This is contradicts the commutativity requirement imposed on the representation $V$.

  Symmetrically, in the case that $I=0$ we set $J$ unmatched, and it can be checked that $J$ has length $<2\epsilon$.

  Otherwise, if both $I$ and $J$ are nonzero, we define $\sigma_L(I) = J$.
  By Lemma~\ref{lem:corronlyif}, putting $I=I[x,y]$ and $J=I[s,t]$, we have $|x-s|\leq \epsilon$, $|y-t|\leq \epsilon$ and the pair of $I,J$ satisfies the Condition \eqref{eq:conast} if they are both short ($|y-x|< 2\epsilon$ and $|t-s|< 2\epsilon$).

  The above arguments show that $\sigma_L$ is indeed an essential $\epsilon$-matching from $B(M)$ to $B(N)$.
\end{itemize}

By definition, 
the maps $F$ and $G$ are mutually inverse bijective maps, thus the claim follows.
\end{proof}

\begin{remark}
By Lemma~\ref{lem:corrif} and the proofs of Lemma~\ref{lem:corronlyif} and Theorem~\ref{thm:corr}, it turns out that for any (not necessarily essential) $\epsilon$-matching $\sigma$, 
we can define an interval decomposable representation $F'(\sigma)$ of $\lace_{\epsilon} \mathbb{Z}$ as follows:
\begin{equation}
  \label{eq:nonessential_matching}
  \begin{array}{rcl}
    F'(\sigma) &:=&
                    \left(\bigoplus\limits_{\substack{\sigma(I)=J\\ \text{ with Condition~\eqref{eq:conast}}}} L_{I,J}\right) \oplus
    \left(\bigoplus\limits_{\substack{\sigma(I)=J\\ \text{ without Condition~\eqref{eq:conast}}}} (L_{I,0}\oplus L_{0,J})\right) \oplus 
    \\
               && \left(\bigoplus\limits_{I \in B(M) \text{, unmatched}} L_{I,0} \right)\oplus
                  \left(\bigoplus\limits_{J\in B(N)\text{, unmatched}} L_{0,J} \right).
  \end{array}
\end{equation}
However, in general, $G(F'(\sigma))\not = \sigma$ since the $I$ and $J$ appearing as the second term in Equation~\eqref{eq:nonessential_matching}, while matched in $\sigma$, are unmatched in $G(F'(\sigma))$.

Suppose that two persistence modules $M$ and $N$ are $\epsilon$-interleaved. This means that we have an object $(M,N,\phi,\psi)\in \operatorname{Int}_{\epsilon} (\mathbb{Z},\vect_{K})$ which corresponds to a representation $V_{M,N}$ by Theorem~\ref{thm:shoelace}. On the other hand, the algebraic stability theorem \cite{bauer2014induced} implies that there is an $\epsilon$-matching $\sigma$ between $B(M)$ and $B(N)$. From this $\epsilon$-matching, we get the interleaving $F'(\sigma)$ expressed as a representation.  How does this compare to the original interleaving $V_{M,N}$?
Theorem~\ref{thm:twist_interleaved_interleavings} provides an answer: $V_{M,N}$ and $F'(\sigma)$ are in fact $\wtilde{\Lambda_{\epsilon}}$-interleaved.

\end{remark}




\section*{Acknowledgements}
Killian Meehan is supported in part by JST CREST Mathematics (15656429). Michio Yoshiwaki was partially supported by Osaka City University Advanced Mathematical Institute (MEXT Joint Usage/Research Center on Mathematics and Theoretical Physics).
Conflict of Interest: The authors declare that they have no conflicts of interest.
\bibliographystyle{plain}
\bibliography{refs}

\end{document}
